\documentclass[12pt]{article}
\usepackage{amsgen, mathtools, amsmath, amsfonts, amsthm, amssymb, enumerate,amscd,hyperref,graphicx}
\usepackage[new]{old-arrows}
\usepackage{tikz-cd}

\newtheorem{defn}{Definition}[section]
\newtheorem{prop}[defn]{Proposition}
\newcommand{\legendre}[2]{\genfrac{(}{)}{}{}{#1}{#2}}

\definecolor{blue}{rgb}{0,0,1}
\definecolor{red}{rgb}{1,0,0}
\definecolor{green}{rgb}{0,.6,.2}
\definecolor{purple}{rgb}{1,0,1}

\long\def\red#1\endred{\textcolor{red}{#1}}
\long\def\blue#1\endblue{\textcolor{blue}{#1}}
\long\def\purple#1\endpurple{\textcolor{purple}{ #1}}
\long\def\green#1\endgreen{\textcolor{green}{#1}}

\newtheorem{lem}[defn]{Lemma}
\newtheorem{thm}[defn]{Theorem}

\newtheorem*{conj*}{Conjecture}
\newtheorem{cor}[defn]{Corollary}

\newcommand {\ZZ}{{\mathbb Z}}

\newcommand {\C}{{\mathbb C}}

\newcommand {\G}{{\Gamma}}

\newcommand {\OO}{{\mathcal O}}

\newcommand {\g}{{\gamma}}

\newcommand {\HH}{{\mathfrak  H}}

\newcommand {\MH}{\mathcal{M}}
\newcommand {\M}{{\mathcal M}}

\newcommand {\RR}{{\mathcal R}}

\newcommand {\MI}{{\mathcal{MI}}}

\def\mod{\operatorname{mod}}

\title{}
\author{ \\
}

\begin{document}

\title{Modular iterated integrals associated with cusp forms}
\date{}
\author
{Nikolaos Diamantis
\\
(University of Nottingham)
}

\maketitle
\begin{abstract}
We construct an explicit family of modular iterated integrals which involves cusp forms. This leads to a new method of producing  "invariant versions" of iterated integrals of modular forms. The construction will be based on an extension of higher-order modular forms which, in contrast to the standard higher-order forms, applies to general Fuchsian groups of the first kind and, as such, is of independent interest.
\end{abstract}

\section{Introduction}

This paper deals with two classes of functions that have not been previously studied together, namely, \emph{modular iterated integrals} and \emph{higher order modular forms}. We show that they are interrelated in a way that key features of one of them can be elucidated through constructions in the other. 

The first class of objects, modular iterated integrals, were introduced recently \cite{BrI, BrII, BrIII} by F. Brown in the context of the theory of \emph{real-analytic modular forms}. Those are real-analytic functions $f$ on the upper half-plane $\mathfrak H$, characterised by \\
i. a transformation law of the form
$$f(\g z)=(cz+d)^r (c \bar z+d)^s f(z), \qquad \text{for all $z \in \mathfrak H,$}$$
and for all $\g = \left ( \begin{smallmatrix} * & * \\ c & d \end{smallmatrix} \right )$ in a suitable group $\G$ and \\
ii. a Fourier series of the form
$$
 f(z)=\sum_{|j| \le M}y^j\left ( \sum_{\substack{m, n \ge 0}} a_{m, n}^{(j)} q^m \bar q^n \right )
$$
for some $m \in \mathbb N$ and $q=e^{2 \pi i z}$. (Precise definition given in Section \ref{basics} )
The motivation for introducing them included their possible use towards arithmetic questions involving periods and evidence  that the modular graph functions of String Theory are real-analytic modular forms. 

A special subclass of the class of real-analytic modular forms consists of the spaces $\MI_{\ell}$  of \emph{modular iterated integrals of length $\ell$}. Their defining relation is 
\begin{equation}\label{MI} \partial \MI_\ell  \subset \MI_\ell+M[y] \times \MI_{\ell-1} \qquad \text{and} \, \, \, 
\bar \partial \MI_\ell  \subset \MI_\ell+\overline{ M}[y]\times \MI_{\ell-1} \end{equation}
where $M[y]$ (resp. $\bar M[y]$) denotes polynomials in $y=$Im$(z)$ with coefficients in the space of standard holomorphic (resp. anti-holomorphic) modular forms and $\partial, \bar \partial$ are certain differential operators that will again be defined precisely in Section \ref{basics}.

A reason for the special interest of this class is that the modular graph functions are expected to belong to it (cf. Sect. 1 of \cite{BrI} and its Sect. 9 where explicit evidence of this is provided). A second reason is that its structure seems to have arithmetic significance, as indicated by evidence provided by Brown (in \cite{BrI, BrII}) for the motivic nature of the space and by the association to it of classical number theoretic invariants, such as L-functions and period polynomials, by J. Drewitt and the author \cite{DD}.
 
In this paper, we will address two questions that arise from the works mentioned above:

{\it Question 1.} The elements of $\MI_{\ell}$ almost exclusively studied in the papers above are those that satisfy a condition more specific than \eqref{MI}, namely
\begin{equation}\label{MIE} \partial \MI_\ell  \subset \MI_\ell+E[y] \times \MI_{\ell-1} \qquad \text{and} \, \, \, 
\bar \partial \MI_\ell  \subset \MI_\ell+\overline{ E}[y]\times \MI_{\ell-1} \end{equation}
where $E$ is the \emph{subspace of $M$ generated by Eisenstein series.}

What can be said about the remaining modular iterated integrals, i.e. the part originating in cusp forms? This is important not only because an answer describes more completely the structure of $\MI_{\ell}$, but, especially, because arithmetic information is normally expected to be encapsulated by forms that are cuspidal. This is particularly relevant in view of the evidence for the arithmetic significance of $\MI_{\ell}$ mentioned above.

We will provide an answer to this question by constructing (in Sect. \ref{explsubc}) an explicit family of such functions originating in cusp forms. To this end, we will first restate (in Section \ref{motdef}) the question in a concretely and precise form, a task of independent interest.

{\it Question 2.} A more explicit characterisation of the space $\MI_{\ell}$ can be given in terms of \emph{$\G$-invariant} linear combinations of
real and imaginary parts of iterated integrals of modular forms. This is proven in the case of elements of $\MI_{\ell}$ originating in Eisenstein series (\cite{BrII}) and is conjectured to hold in general. Constructing such ``invariant versions" of iterated integrals of modular forms is one of the important themes of \cite{BrI}, especially in relation to the applications to the theory of modular graph functions. 

Here we discuss a new approach to this problem: We construct a family $\{\psi^{\pm}_{h; r, s}\}$ of explicit ``real-analytic iterated integrals" whose invariant piece is built very naturally into each one of them and, further, belongs to an extension of $\MI_{\ell}.$ This is done, in the case of $\ell=2$,  in Sect. \ref{HOFS1}, where it is shown how $\psi^{\pm}_{h; r, s}$ yield 
polynomials $\phi_{r, s}^{\pm}(h; -)$ whose coefficients are $\G$-invariant elements of (an extension of) $\MI_{\ell}$. 

The tool with which we achieve both of those two aims is based on \emph{higher order forms}, the second  object we deal with in this work.

The characterising feature of higher order modular forms, in the special case of order $2$ and weight $k$, for instance, is the transformation law
$$f|_k\g \delta-f|_k\g-f|_k\delta+f=0, \qquad \text{for all $\g, \delta \in \G$}$$
where the action of the group on the function is given by 
$$g|_k \g(z)=g(\g z) (cz+d)^{-k}.$$
The precise definition will be given in Sect. \ref{HOFS}, where the original notion of higher-order forms will be generalised.  Higher order modular forms have been studied from various perspectives (analytic, adelic, algebraic, spectral etc.) and led to applications to the theory modular symbols, mathematical physics etc. In all these cases, the theory had to be developed on congruence subgroups of level higher than $1$, because such forms were parametrised by cusp forms of weight $2$, which are trivial in level $1$. This was a unnatural constraint because it excluded integrals of higher weight forms from consideration and it also prevented availing oneself of simplifications occuring in SL$_2(\mathbb Z)$. 

In this paper, we resolve this problem too, by proposing a very general framework within which to consider higher order forms. The class obtained includes several known and new objects, including, higher order forms for all levels and iterated integrals.

It turns out that the solution to this problem allows us to realise the constructions behind our answers to Questions 1 and 2 above.
Firstly, it allows us to produce a family of (extended) modular iterated integrals of length 2  originating in cusp forms (Question 1). These modular iterated integrals, in turn, by their very construction, are obtained from  a class of second-order modular forms whose prototypes are exactly the iterated integrals of cusp forms (Question 2). At the same time, these second order modular forms are not of an ad hoc nature. They form a basis of the class of second-order modular forms they belong to (Th. \ref{1stclassif}). This suggests a deeper relation between the two objects that are the subject of this paper.

{\bf Acknowledgements.} The author is grateful to F. Brown, L. Candelori, C. Franc, G. Mason and F. Str\"omberg for many helpful comments and suggestions. Research on this work was  supported in part by 
by EPSRC grant EP/S032460/1.

\section{Subclasses of the space of real-analytic modular forms }
\subsection{Review of definitions and notation}\label{basics}
We start by introducing some of the notation we will be using and by recalling the definitions of real-analytic modular forms and of modular iterated integrals.

\subsubsection{Basic spaces and actions} 
Let $\G=$SL$_2(\ZZ)$ and set $S=\left ( \begin{smallmatrix} 0 & -1 \\ 1 & 0 \end{smallmatrix} \right )$, 
$T=\left ( \begin{smallmatrix} 1 & 1 \\ 0 & 1 \end{smallmatrix} \right )$ and 
$R=ST=\left ( \begin{smallmatrix} 0 & -1 \\ 1 & 1 \end{smallmatrix} \right ).$  

If $\HH$ denotes the upper half-plane and $z=x+iy$, set
$$\RR:=\{\text{real analytic $f$}: \HH \to \C;  f(z)=O(y^C) \, \, \text{as $y \to \infty$, uniformly in $x$, for some $C>0$}\} $$
$$\RR_c:=\{f \in \RR;  \text{for all $c>0,$ $f(z)=O(e^{-cy})$ as $y \to \infty$, uniformly in $x$}\}$$
$$\OO:=\{\text{holomorphic} \, \, f \in \RR \}$$
$$\OO_c:=\{\text{holomorphic} \, \, f \in \RR_c \}$$
Suppose that $r, s$ are positive integers of the same parity such that $r+s \ge 4$ and that $k$ is an even positive integer. For a $f \in \RR$ and $\g \in \G$, define a function $f\underset{r, s}{|}  \g$ by 
$$f \underset{r, s}{|} \g(z)=j(\g, z)^{-r} j(\g, \bar z)^{-s }f(\g z) \qquad \text{for all $z \in \HH$ }.$$
Here
$$j(\g, z)=c_{\g}z+d_{\g} \qquad \text{where $ \g=\left ( \begin{matrix} a_{\g} & b_{\g} \\ c_{\g} & d_{\g} \end{matrix} \right )$}.$$
We extend the action to $\C[\G]$ by linearity.

We can now give the definition of a real-analytic modular forms. 

We call an $f \in \RR$ (resp. $f \in \RR_c$) a \emph{real-analytic modular \emph{(resp.} cusp \emph{)} form of weights $(r, s)$ for $\G$} if \newline
1. for all $\g \in \G$ and $z \in \HH$, we have $f \underset{r, s}{|} \g=f$, i.e.
$$f(\g z)=j(\g, z)^r j(\g, \bar z)^s f(z)  \qquad \text{for all $z \in \HH$, }$$
\noindent
2.  for some $M \in \mathbb N$ and $a_{m, n}^{(j)} \in \C$.
\begin{equation} \label{FE} f(z)=\sum_{|j| \le M}y^j\left ( \sum_{\substack{m, n \ge 0
}} a_{m, n}^{(j)} q^m \bar q^n \right ) \qquad \quad (q:=\exp(2 \pi i z)) \end{equation}

We denote the space of real analytic modular (resp. cusp) forms of weights $(r, s)$ for $\G$ by $\MH_{r, s}$
(resp. $\mathcal S_{r, s}$). We set $\MH=\bigoplus_{r, s }\MH_{r, s}$ (resp. $\mathcal S=\bigoplus_{r, s }\mathcal S_{r, s}.$)

For $s=0$, upon restriction to holomorphic functions, we retrieve the space of standard holomorphic modular (resp. cusp) forms denoted by $M_r$ (resp. $S_r$). We also set $M=\bigoplus_{r}M_{r}$ (resp. $S=\bigoplus_{r} S_{r}.$)

\subsubsection{Lie structure} 
 The Lie algebra $\mathfrak{sl}_2$ acts on $\MH$ via the Maass operators $\partial_{r}: \MH_{r, s} \to \MH_{r+1, s-1}$
and $\bar \partial_s: \MH_{r, s} \to \MH_{r-1, s+1}$ given by
$$\partial_r=2iy\frac{\partial}{\partial z}+r \qquad \text{and} \, \, \,  \bar \partial_{s}=-2iy\frac{\partial}{\partial \bar z}+s.$$
(For proofs and further details on this and the rest of this subsection, see Sect. 2.2 of \cite{BrI})
These operators induce bigraded derivations on $\MH$ denoted by $\partial$ and $\bar \partial$ respectively. We set
$$\partial^{(m)}:=\partial \circ \partial \dots \circ \partial: \MH_{r, s} \to \MH_{r+m, s-m}$$
with a similar definition for $\bar \partial^{(m)}$.

Two pairs of identities we will be using often, taken from \cite{BrI} (Lemma 2.5 and (2.13)) are
\begin{align} \partial_r  ( g \underset{r,s}{|} \g  )=\left (  \partial_r  g \right )  \underset{r+1,s-1}{|} \g \qquad & \text{and}
\,  \, \, \bar \partial_r  ( g \underset{r,s}{|} \g  )=\left ( \bar  \partial_r  g \right )  \underset{r-1,s+1}{|} \g
\label{equiv:1}\\
\partial_r  ( y^{k}g  )=y^k \partial_{r+k}( g )   \qquad & \text{and}
\,  \, \, \bar \partial_s  ( y^{k}g  )=y^k \bar \partial_{s+k}( g )  
\label{equiv:2}
\end{align}
To simplify notation, we will omit the index of $\partial_r$ (resp. $\bar \partial_s$), when this is implied by the context.

\subsubsection{Modular iterated integrals} \label{SMIdef}
We can now recall the definition of the space $\MI_{\ell}$ of \emph{modular iterated integrals of length $\ell$}. Set recursively:\\
$\bullet$ $\MI_{-1}=0$ \\
$\bullet$ For each integer $\ell \ge 0,$ we let $\MI_\ell$ be the largest subspace of $\underset{r, s \ge 0}{\bigoplus}\mathcal M_{r, s}$ which satisfies
\begin{align} \partial \MI_\ell & \subset \MI_\ell+M[y] \times \MI_{\ell-1} \label{MIdef:1}\\
\bar \partial \MI_\ell & \subset \MI_\ell+\overline{ M}[y]\times \MI_{\ell-1} \label{MIdef:2}
\end{align}
where $\overline{M}$ is the ring of anti-holomorphic modular forms.

In Lemma 3.10 of \cite{BrI} it is proved that $\MI_0=\mathbb C[y^{-1}]$ and, in Cor. 4.4 of \cite{BrI}, the following statement:
\begin{prop}\label{MI1}
Let $\mathcal E_{r, s}$ be the real-analytic Eisenstein series normalised so that,  for $r, s \in \mathbb N$ and $z \in \mathfrak H$, $$\mathcal E_{r, s}(z)=\sum_{\g \in B \backslash
\G} \frac{1}{j(\gamma, z)^{r} j(\gamma, \bar z)^{s}}.
$$
(Here $B=\{\pm T^n; n \in \mathbb Z\}$.) Then
\begin{equation}  \MI_1=\mathbb C[y^{-1}] \otimes \bigoplus_{r,s \ge 1, r+s \ge 4} \mathbb C y \mathcal E_{r, s}.\end{equation}
\end{prop}
Notice that this normalisation of $\mathcal E_{r, s}$ is different from that of \cite{BrI}.

\subsection{The space of extended modular iterated integrals}
\subsubsection{Motivating remarks and definition}\label{motdef}
With Prop. \ref{MI1} and the definition of $\MI_{\ell}$ we see that the space $\MI_2$ is defined as the largest subspace of $\underset{r, s \ge 0}{\bigoplus}\mathcal M_{r, s}$ which satisfies
\begin{align} \partial \MI_2 & \subset \MI_2+ \underset{\substack{j \in \ZZ \\  r, s \ge 1, r+s \ge 4}}{\bigoplus} y^j \mathcal E_{r, s} M \label{MI1def:1}\\
\bar \partial \MI_2 & \subset \MI_2+ \underset{\substack{j \in \ZZ \\  r, s \ge 1, r+s \ge 4}}{\bigoplus} y^j \mathcal E_{r, s} \bar M. \label{MI1def:2}
\end{align}
Section 9 of \cite{BrI} contains several important, explicit examples of elements of $\MI_2$ that correspond to the subspace of $M$ generated by (holomorphic) Eisenstein series in the RHS of $\eqref{MI1def:1}$ and $\eqref{MI1def:2}$. We would like to investigate the part corresponding to the remaining ``cuspidal piece". Specifically, we consider 
the largest subspace $\mathcal N$ of $\mathcal M$ which satisfies
\begin{equation}\label{preMI2'} 
\partial \mathcal N 
 \subset \mathcal N+ \underset{\substack{j \in \ZZ \\  r, s \ge 1, r+s \ge 4}}{\bigoplus} y^j \mathcal E_{r, s} S 
\qquad \text{and} \quad  \bar \partial \mathcal N 
 \subset \mathcal N+ \underset{\substack{j \in \ZZ \\  r, s \ge 1, r+s \ge 4}}{\bigoplus} y^j \mathcal E_{r, s} \bar S. 
\end{equation}
The motivation for that, apart from the general aim of classifying the space $\MI_2$, is to study modular iterated integrals whose L-functions are more likely to have classical arithmetic significance than those originating in Eisenstein series.

However, in contrast to the definition of $\MI_{\ell}$ (Sect. \ref{SMIdef}) it does not seem possible to impose to this definition the restriction that the space should be contained in the ``first quadrant" $\underset{r, s \ge 0}{\bigoplus}\mathcal M_{r, s}$ only. 

We provide a heuristic argument why this is not possible. Assume the ``cuspidal part" were indeed restricted to the ``first quadrant" and consider $f \in \M_{r, s}$ satisfying \eqref{preMI2'}. Set 
$F:=\partial^{(s)}f.$ Then, since $\partial_{s+r}F \in \M_{r+s+1, -1}$, we should have
$$\partial_{s+r}F = \sum y^j \mathcal E_{m, l} g$$
where the sum ranges over a finite number of $j \in \mathbb Z$, $m, l \ge 1$ with $m+l \ge 4$ and $g \in S$. With \eqref{equiv:2}, this gives
$\partial_0(y^{r+s} F)=\sum y^{r+s+j} \mathcal E_{m, l} g$ and hence, since $g$ is cuspidal,
$$F=\frac{y^{-s-r}}{2} \sum y^j \int_{\infty}^0 (y+t)^{r+s+j-1} \mathcal E_{m, l}(z+it) g(z+it) dt+y^{-r-s}\overline{h(z)}$$
for some holomorphic function $h(z)$. Since, again according to our assumption, $F$ belongs to a space contained in the ``first quadrant", $\bar \partial^{(r+1)}F$ should belong to $\oplus y^j \mathcal E_{r, s} \bar S.$ 
On the other hand, the recursive relations of Prop. 4.1 of \cite{BrI}, combined with \eqref{equiv:2}, show that $\bar \partial^{(r+s+1)}F$ is a linear combination of elements of the form
$$y^j \int_{\infty}^0 t^{j'} \mathcal E_{m, l}(z+it) g(z+it) dt \qquad \text{and} \quad y^{-r-s}\overline{\partial^{(r+s+1)} h(z)}$$
where, for compactness of notation, we have taken $\mathcal E_{0, m}$ to stand for $y \mathbb G_{m+2}$ in accordance to  (4.1) of \cite{BrI}. Such linear combinations do not seem to belong to $\oplus y^j \mathcal E_{r, s} \bar S.$ 

With this in mind, we introduce a variant of the definition of $\MI_2$:
\begin{defn}\label{MI2'}  We let the space $\MI'_2$ of \emph{extended modular iterated integrals of length $2$}  be the largest subspace of $\mathcal M$ which satisfies
\begin{equation}\label{eqMI2'} 
\partial \MI'_2 
 \subset \MI'_2+ \underset{\substack{j \in \ZZ \\  r, s \ge 1, r+s \ge 4}}{\bigoplus} y^j \mathcal E_{r, s} S 
\qquad \text{and} \quad  \bar \partial \MI'_2 
 \subset \MI'_2+ \underset{\substack{j \in \ZZ \\  r, s \ge 1, r+s \ge 4}}{\bigoplus} y^j \mathcal E_{r, s} \bar S.
\end{equation}
\end{defn}
That is, the space $M$ in RHS has been replaced by $S$ and the restriction of the space belonging to the ``first quadrant" is no longer required.

\subsubsection{An explicit sub-class of $\MI'_2.$}\label{explsubc}
We will now define an explicit family of elements of $\MI'_2$ that will give an answer to Question 1. In later sections, we will show that it originates in second-order modular forms. We will first introduce some preparatory constructions and results. 

Let $P_{k-2}$ denote the space of polynomials in $\mathbb C[X]$ of degree $\le k-2$, acted upon by $\underset{2-k, 0}{|}$. 
Denote the tensor product of the representations $(\RR,  \underset{r, s}{|})$ and $(P_{k-2},  \underset{2-k, 0}{|})$ by $ \underset{r, s, 2-k}{|}$.  It acts on $\mathcal \RR \otimes P_{k-2}$ as 
$$(f \underset{r, s,2- k}{|} \gamma)(z, X)=f(\gamma z, \gamma X) j(\gamma, z)^{-r} j(\gamma, \bar z)^{-s}j(\gamma, X)^{k-2}.$$
We use the same notation for the sub-representations corresponding to $\RR_c, \OO$ and $\OO_c.$

Let now 
$$f(z)=\sum_{n=1}^{\infty} a(n) e^{2 \pi i nz}$$ 
be a cusp form of weight $k$ for $\G$ and consider its Eichler integrals
$$F_f^+(z, X)=\int_{i\infty}^z f(w)(w-X)^{k-2}dw \qquad \text{and} \, \, 
F_f^-(z, X)=\overline{\int_{i\infty}^z f(w)(w-X)^{k-2}dw}$$
where the bar means complex conjugation (acting trivially on $X$).
We set
$$r_f(\g; X):=r^+_f(\gamma ; X)=\int_{\gamma^{-1} i \infty}^{i \infty} f(w)(w-X)^{k-2}dw \qquad 
r^-_f(\gamma; X)=\overline{\int_{\gamma^{-1} i \infty}^{i \infty} f(w)(w-X)^{k-2}dw}$$ 
and
\begin{equation}\label{defphi}
\phi^{\pm}_{r, s}(f; z, X):=\phi^{\pm}_{r, s}(z, X)=\sum_{\g \in B \backslash \G} F_f^{\pm}  
\underset{r,s, 2-k}{\big |} \g=\sum_{\g \in B \backslash \G} \frac{
F_f^{\pm}(\g z, \g X)}{j(\gamma, z)^{r} j(\gamma, \bar z)^{s}}j(\g, X)^{k-2}.
\end{equation}
We have the following proposition.
\begin{prop}\label{phi} Suppose that $r+s>k.$ Then, for each $z \in \mathfrak H$, the series $\phi^{\pm}_{r, s}(z, X)$ converges
absolutely and it is invariant under the action of $ \underset{r,s, 2-k}{|}$ of $\G$. Its
polynomial coefficients are of (at most) polynomial growth at infinity, 
\end{prop}
\begin{proof} We show it for $\phi^+$, the proof for $\phi^-$ being deduced upon conjugating $\phi^+$.

By the first equality of \eqref{defphi}, $\phi^+_{r, s}$  is invariant under the action $ \underset{r,s, 2-k}{|}$ of $\G$, for those $r, s$ for which it converges.

To prove the statement about absolute convergence, we first note that
the change of variables $w \to \gamma w$, the transformation law of $f$ and 
the identity \begin{equation}\label{gzgX}
(\g z- \g X) j(\g, z) j(\g, X)= z-X 
\end{equation}
imply that $F_f^+(\g z, \g X)j(\g, X)^{k-2}$ equals
\begin{align} & \int_{\g^{-1} i \infty}^z f(w)(w-X)^{k-2} dw=
r_f(\g; X) +F^+(z, X) = \label{decomp:1}\\
& \sum_{j=0}^{k-2} (-1)^j \binom{k-2}{j}
\int_{\g^{-1} i \infty}^{i \infty}f(w)(w- \g^{-1} \infty)^j dw  \cdot (X-\g^{-1} \infty)^{k-2-j}+F^+(z, X) . \label{decomp:2}
\end{align}
By applying this decomposition to the defining series for $\phi^+_{r,s}$ we get a sum of two terms:

To analyse the part corresponding to the first term of \eqref{decomp:2}, 
we note that each of the integrals appearing in the sum 
 is (up to a power of $i$) the value at $s=l+1$ of the ``completed" L-function with additive twists. Specifically, 
\begin{equation}\label{lambda}
\Lambda_f \left (s, \frac{p}{q} \right ):=\int_0^{\infty} f \left (\frac{p}{q}+ix \right )x^{s-1} dx=\Gamma(s) (2 \pi)^{-s} \sum_{n=1}^{\infty}\frac{a(n)e^{2 \pi i np/q}}{n^s}.
\end{equation}
It is well-known that $\Lambda_f(s, p/q)$ has a functional equation (see \cite{DHKL} for a general version of the functional equation) and, with convexity, this implies that, for each $j=0, \dots, k-2,$ 
\begin{equation}\label{conv}
q^{j+1} \Lambda _f(j+1, p/q) \ll q^{k-1+\epsilon}.
\end{equation} 
Also, \begin{equation}\label{tri}
X-\g^{-1} i \infty=\left (X-z \right )+(z-g^{-1}i \infty)=\left (X-z \right )+j(\g, z)/c_{\g}.
\end{equation}
With these observations and the binomial formula applied on \eqref{tri}, we can bound the coefficient of $(X-z)^{k-2-j-m}$, ($0 \le m \le k-2-j$) in the polynomial
$$\sum_{\g \in B \backslash \G} \frac{\int_{\g^{-1} i \infty}^{i \infty}f(w)(w- \g^{-1} \infty)^j dw  \cdot (X-\g^{-1} \infty)^{k-2-j}}{j(\gamma, z)^{r} j(\gamma, \bar z)^{s}}.$$
This coefficient equals 
 $$\binom{k-2-j}{m} i^{j+1}
\sum_{\g \in B \backslash \G}  \frac{\Lambda_f(j+1, \g^{-1} \infty)}{c_{\g}^m j(\gamma, z)^{r-m} j(\gamma, \bar z)^{s}}
\ll \sum_{\g \in B \backslash \G}  \frac{c_{\gamma}^{k-2-j-m}}{|j(\gamma, z)|^{r+s-m}}
$$
The elementary inequality $|c_{\g}|\le |j(\g, z)| \text{Im}(z)^{-1} $ implies that the sum  is 
\begin{equation}
\label{growth}
\le  y^{m+(2-k-r-s+j)/2}\sum_{\g \in B \backslash \G}  \text{Im}(\g z)^{(r+s+j-k+2)/2}
\end{equation}
which converges for $r+s+j-k+2>2.$ 

The second term of \eqref{decomp:2}, gives 
$$\mathcal E_{r, s}(z) \int_{i \infty}^{z}f(w)(w-X)^{k-2} dw $$
which, by comparison with the standard non-holomorphic Eisenstein series, converges absolutely if $r+s>2$.

Therefore, both pieces of $\phi^+_{r, s}$ induced by \eqref{decomp:2} will converge, if $r+s+j-k+2>2,$ for all $j=0, \dots k-2$ and $r+s>2$. This is indeed the case if $r+s>k$. 

The bound \eqref{growth} and the polynomia growth of $E_{r, s}(z)$ show that the coefficients of $(X-z)^{j}$ (and of $X^j$) in $\phi^{+}_{r, s}$ are of, at most, polynomial growth as $y \to \infty.$

\end{proof}
The series $\phi^{\pm}_{r, s}(z, X)$ can be decomposed in terms of elements of $\mathcal M.$ Specifically, 
let $\phi^{\pm}_{r, s}(f; i, z)=\phi^{\pm}_{r, s}(i, z)$ ($i=0, \dots, k-2$) be functions such that
\begin{equation}\label{expansion}
\phi^{\pm}_{r, s}(z, X)=\sum_{i=0}^{k-2}\phi^{\pm}_{r, s}(f; i, z) (X-z)^i (X-\bar z)^{k-2-i}.
\end{equation}
From Prop. 7.1 of \cite{BrI}, we know that $\phi^{\pm}_{r, s}(f; i, z)$ is $\underset{r+i, s+k-2-i}{|}$-invariant. To show that it actually belongs to $\mathcal M_{r+i,s+k-2-i}$ we need to show that it has a Fourier expansion of the form \eqref{FE}. This is part of the content of the next proposition.

\begin{prop} \label{Mrs}For each $j=0, \dots, k-2$, we have 
\begin{multline}\label{addit}
\phi_{r, s}^{+}(j; z)=(-1)^j \binom{k-2}{j} y^{2-k} \left ( \int_{ i \infty}^z f(w) (w-\bar z)^j(w-z)^{k-2-j}dw \right ) \mathcal E_{r, s}+\\ 
\sum_{m=0}^j \sum_{n=0}^{k-2-j} \alpha_{m, n} y^{2-k}
\sum_{1 \ne \g \in B \backslash \G} \frac{\Lambda_f(m+n+1, \g^{-1}(\infty)) c_{\g}^{m+n-k+2}}{j(\g, z)^{r+j+n+2-k} j(\g, \bar z)^{s+m-j}}
\end{multline}
and
\begin{multline}\label{addit-}
\phi_{r, s}^{-}(j; z)=(-1)^j \binom{k-2}{j} y^{2-k} \left ( \overline{\int_{ i \infty}^z f(w) (w-\bar z)^{k-2-j} (w-z)^{j}dw} \right ) \mathcal E_{r, s}+\\ 
\sum_{m=0}^j \sum_{n=0}^{k-2-j} \alpha_{m, n} y^{2-k}
\sum_{1 \ne \g \in B \backslash \G} \frac{\overline{\Lambda_f(m+n+1, \g^{-1}(\infty))} c_{\g}^{m+n-k+2}}{j(\g, z)^{2-k+r+m+j} j(\g, \bar z)^{s-j+n}}
.
\end{multline}
where $\alpha_{m, n}:=i^{1-2j-m-n} \binom{k-2}{j} \binom{j}{m} \binom{k-2-j}{n}.  $
Further, each $\phi_{r, s}^{\pm}(j; z) \in \mathcal M$.
\end{prop}
\begin{proof} 
Replacing $w-X$ in $F_f^+(z, X)$ according to the identity
$$w-X=\left ( (w-z)(X-\bar z)+(\bar z-w)(X-z)\right )/\text{Im} z$$
and expanding with the binomial theorem we see that the coefficients in the RHS of \eqref{expansion} can be written as
$$\phi^{+}_{r, s}(j; z)=(-1)^j \binom{k-2}{j}
\sum_{\g \in B \backslash \G} \left ((\text{Im}z)^{2-k} 
\int_{ i \infty}^z f(w) (w-\bar z)^j(w-z)^{k-2-j}dw \right )  \underset{r+j,k-2+s-j, 0}{|} \g. $$
(resp. 
$$\phi^{-}_{r, s}(j; z)=(-1)^j \binom{k-2}{j}
\sum_{\g \in B \backslash \G} \left ((\text{Im}z)^{2-k} 
\overline{\int_{ i \infty}^z f(w) (w-\bar z)^{k-2-j}(w-z)^{j}dw} \right )  \underset{r+j,k-2+s-j, 0}{|} \g. )$$
Upon unraveling the definition of the action $|$, the sum in $\phi^{+}_{r, s}(j; z)$ equals
$$\sum_{\g \in B \backslash \G} \frac{y^{2-k} \int_{ i \infty}^{\g z} f(w) (w-\g \bar z)^j j(\g, \bar z)^j (w-\g z)^{k-2-j}j(\g, z) ^{k-2-j} dw}{j(\g, z)^r j(\g, \bar z)^s} 
.$$
The change of variables $w \to \g w$ and the transformation law of $f$ imply that the integral equals:
\begin{multline} \label{phidecom} \int_{ \g^{-1} i \infty}^z f(w) (w-\bar z)^j(w-z)^{k-2-j}dw=
\left ( \int_{ \g^{-1} i \infty}^{i \infty} +\int_{ i \infty}^z \right ) f(w) (w-\bar z)^j(w-z)^{k-2-j}dw
\end{multline}
The first integral in the RHS is $0$ for $\g=1$. For $\g \ne 1,$  the binomial theorem leads to
$$\sum_{m=0}^j \sum_{n=0}^{k-2-j} \binom{j}{m} \binom{k-2-j}{n} \frac{j(\g, \bar z)^{j-m}
 j(\g,  z)^{k-2-j-n}}{(-c_{\g})^{k-2-m-n}}  \int_{ \g^{-1} i \infty}^{i \infty}
f(w) (w-\g^{-1} i \infty)^{m+n}dw.$$
Equation \eqref{addit} follows from this and \eqref{lambda} combined with \eqref{phidecom}. Equation \ref{addit-} can be deduced from \eqref{addit} upon a conjugation. 

To show that $\phi^+_{r, s}(j; z)$ has a Fourier expansion of the form \eqref{FE}, we apply the usual double coset decomposition to the series 
$$\sum_{1 \ne \g \in B \backslash \G} \frac{\Lambda_f(m, \g^{-1}(\infty)) c_{\g}^{n}}{j(\g, z)^{p} j(\g, \bar z)^{t}} \qquad \text{(where $m, n, p, t$ are integers).}$$
Then, this becomes
\begin{multline}\label{explFour}\sum_{c>0} \sum_{d \mod c} \sum_{l \in \mathbb Z} \frac{\Lambda_f(m, -\frac{d}{c}) c^n}{(c(z+l)+d)^{p} (c(\bar z+l)+d)^{t}}=\\
\sum_{c>0} c^{n-p-t}\sum_{d \mod c} \Lambda_f(m, -\frac{d}{c}) \sum_{l \in \mathbb Z} e^{2 \pi i l(x+\frac{d}{c})}
\int_{\mathbb R}\frac{e^{-2 \pi i l t_1} dt_1}{(t_1+iy)^{p} (t_1-iy)^{t}}
\end{multline}
where, for the last equality, we used the Poisson formula followed by a change of variables. With 3.2 (12) of \cite{Erd}, combined with 13.14.9 of \cite{NIST}, the integral equals $P_l(y)e^{-2 \pi| l |y}$, for some polynomial $P_l(y)$ of degree $\le |p-2|$ in $y^{\pm}.$
This implies that \eqref{explFour} can be written as
$$\sum_{l \ge 0} q^l P_l(y) 
\sum_{c>0}\sum_{d \mod c} \frac{\Lambda_f(m, -\frac{d}{c}) e^{2 \pi i l \frac{d}{c}}}{ c^{p+t-n}}+
\sum_{l <0} \bar q^{-l} P_l(y) 
\sum_{c>0}\sum_{d \mod c} \frac{\Lambda_f(m, -\frac{d}{c}) e^{2 \pi i l \frac{d}{c}}}{ c^{p+t-n}}.$$ 
We can appplying this with $m, n, p, t$ replaced by $m+n+1, m+n-k+2, r+j+n+2-k, s+m-j$ respectively, because, by \eqref{conv}, the inner series converge with $r+s>k.$  We deduce that $\phi^+_{r, s}(j; z)$ has an expansion of the form \eqref{FE}. The analogous assertion for $\phi^-_{r, s}(j, z)$ can be deduced from this after a conjugation. 
\end{proof}

We can now verify an identity for $\phi_{r, s}^{\pm}=\phi_{r, s}^{\pm}(f; z, X)$ that will allow us to show that $\phi^{\pm}_{r, s}(j; z) \in \MI'_2$.
\begin{prop} \label{keypr}For each $r, s, k$ as above, we have
\begin{align} \partial_{r}(\phi^{\pm}_{r, s}) &=r \phi^{\pm}
_{r+1, s-1}+
2i \delta y f(z) (X-z)^{k-2} \mathcal E_{r, s} \\
\bar \partial_{s}( \phi^{\pm}_{r, s}) &=sy^j \phi^{\pm}_{r-1, s+1}-2i (1-\delta) y \overline{f(z)} (X-\bar z)^{k-2}  \mathcal E_{r, s} 
\end{align} 
where $\delta$ is $1$ in the plus-case and $0$ otherwise.
\end{prop}
\begin{proof} With \eqref{equiv:1}, we have 
\begin{equation}\label{seriesequiv}\partial_r(\phi^{\pm}_{r, s})=\sum_{\g \in B \backslash \G} \partial_r(F_f^{\pm}) \underset{r+1, s-1, 2-k}{|} \g.\end{equation} The definition of $\partial_r$ and the identity Im$(\g z)=$Im$(z)/(j(\g, z) j(\g, \bar z))$ imply that, in the plus-case, this equals 
\begin{equation} r\phi^+_{r+1, s-1} +2i y \sum_{\g \in B \backslash
\G} \frac{f(\g z) (\g z-\g X)^{k-2}j(\g, X)^{k-2}}{j(\gamma, z)^{r} j(\gamma, \bar z)^{s}}\end{equation} 
With the transformation law for $f(z)$ and \eqref{gzgX}, this implies the statement in this case.

In the minus-case, \eqref{seriesequiv} equals
$ r\phi^-_{r+1, s-1}$
as required.
The second equation is proved upon conjugating the first. 
\end{proof}

We now define our sub-class of $\MI_2'$ as the vector space $\mathcal A$ generated over $\mathbb C$ by $\phi^{\pm}_{r, s}(f; i, -),$ for all $f \in S_k$ ($k \ge 12$), all integers $r, s$ such that $r+s$ is even and $> k,$ and $0 \le i \le k-2$. 
With this definition we can state our answer to Question 1, as follows;
\begin{thm}\label{mathcalA} The space $\mathcal A$ is a subspace of the space $\MI'_2$ of extended modular iterated integrals of  length $2$.
\end{thm}
\begin{proof} An elementary computation implies that for all real-analytic $f_j$, we have 
\begin{equation}\label{coeffs}
\partial_m\left ( \sum_{j=0}^{k-2}f_j(z)(X-z)^j(X-\bar z)^{k-2-j}\right )=
\sum_{j=0}^{k-2}(\partial_{m+j} f_j(z)-(j+1)f_{j+1}(z))(X-z)^j(X-\bar z)^{k-2-j}
\end{equation}
where, for convenience, $f_{k-1}$ is set to equal $0.$

Proposition \ref{keypr} combined with \eqref{coeffs} and its analogue for $\bar \partial$, implies that each $\partial_r \phi^{\pm}_{r, s}(i, -) $ is a linear combination of $\phi^{\pm}_{r, s}(i, -),$ (for varying $r, s, i$)
and an element of $S[y] \times \bigoplus_{r, s} \mathcal E_{r, s}$. Therefore, $\mathcal A$ satisfies the inclusions \eqref{eqMI2'}. 

In the same way, we verify the analogous statement for $\bar \partial_s \phi^{\pm}_{r, s}(i, -) $ with $\bar S$  in place of $S$.

Finally, by Prop. \ref{Mrs}, $\mathcal A$ is a sub-space of $\mathcal M$. Since $\MI'_2$ is, by definition, the largest subspace of $\mathcal M$ satisfying the inclusions \eqref{eqMI2'}, $\mathcal A$ is contained in $\MI'_2$.
\end{proof}

\section{The space of iterated invariants}\label{HOFS}
To define our extended higher-order modular forms, it will be necessary to describe a general framework involving a family of representations. See \cite{Dei1, Dei2} for two alternative general definitions of higher-order objects, which are built on only one representation and which use the formalism of the augmentation ideal. 

Exceptionally, we will give the next definition for general Fuchsian groups $\Gamma$ of the first kind acting on $\mathfrak H$ with non-compact quotient $\G \backslash \mathfrak H$. The reason is that we want to compare it with the previous definition of higher-order modular forms, which are trivial in level $1$.

Let $V=(\rho_i, V_i)_{i \ge 0}$ be a sequence of representations of $\Gamma$, where the \emph{right}-action on each $V_i$ is denoted by $.$ Assume further that the $\mathbb C$-vector spaces $V_i$ are finite dimensional when $i \ge 1.$ For each $n \in \mathbb N$, we consider the tensor representation $\otimes_{i=0}^{n-1} V_i$. To ease notation, we will generally denote the action on it also with $.$ It will generally be clear which representation it refers to in each case, but, in cases of potential ambiguity, it will be explained separately.

In the following definition, if $V$ is a $\Gamma$-module, we view $H^0(\Gamma, V)$ as a subset of $V$.
\begin{defn} Set
$M^{(0)}:=\{0\}$ and define, inductively, $M^{(n)}=M^{(n)}(V)$ to be the subspace of $\otimes_{i=0}^{n-1} V_i$ given by 
\begin{equation*}
M^{(n)}=\emph{pr}^{-1} H^0\left ( \G,  (\otimes_{i=0}^{n-1} V_i) /(M^{(n-1)} \otimes V_{n-1}) \right ) \end{equation*}
where the implied action is induced by that of $\G$ on $\otimes_{i=0}^{n-1} V_i$ and \emph{pr} is the canonical projection of
 $\otimes_{i=0}^{n-1} V_i$ onto $\otimes_{i=0}^{n-1} V_i/(M^{(n-1)} \otimes V_{n-1})$.
We then set
$$M_c^{(n)}=M^{(n)} \cap \bigcap _{\text{parabolic}\, \, \pi} H^0(\langle \pi \rangle,  \otimes_{i=0}^{n-1} V_i) $$
where $\langle \pi \rangle$ is the subgroup generated by $\pi$.
\end{defn}
We call the elements of $M^{(n)}(V)$
{\it iterated invariants of order $n.$}

In the next proposition, we show that this definition is well-founded and we give an equivalent formulation of it:
\begin{prop}\label{stru} (i) For each $n \in \mathbb N,$ $M^{(n)}$ and $M^{(n-1)} \otimes V_{n-1}$
 are closed under the action of $\Gamma$. \\
(ii) We have $M^{(n-1)} \otimes V_{n-1}  \subset M^{(n)}$. Therefore, upon composition with the natural inclusion $M^{(n-1)} \xhookrightarrow{} M^{(n-1)} \otimes V_{n-1},$ we deduce that $M^{(n-1)} \xhookrightarrow{} M^{(n)}$  and
$M_c^{(n-1)} \xhookrightarrow{} M_c^{(n)}$.
 \\
(iii) The space $M_c^{(n)}$ is isomorphic to the space of $f \in \otimes V_i $ such that, for each $\g \in \G$,
$$f.(\gamma-1) \in M^{(n-1)} \otimes V_{n-1}  \qquad \text{and, for each parabolic $\pi \in \G$,} \, \, \, f.(\pi-1)=0$$
\end{prop}
\begin{proof}
(i) (Induction in $n$). Let $f$ be an element of $M^{(n)}$ ($n \ge 1$). Then, by definition, $f.(\varepsilon-1) \in M^{(n-1)} \otimes V_{n-1}$ for each $\varepsilon \in \G$. Let $\g \in \G.$ Then, for each $\delta \in \G,$ $(f.\gamma) .(\delta-1) =g.\gamma $ for $g:=f.(\gamma \delta \gamma^{-1}-1) \in M^{(n-1)} \otimes V_{n-1}$.
Suppose that $g=\sum_{j=0}^{k_{n-1}} f_j \otimes v_j$, for some $f_j \in M^{(n-1)},$ where $\{v_i\}_{j=0}^{k_{n-1}}$ is a basis of $V_{n-1}.$ Then 
\begin{equation}\label{inclu}
(f.\gamma) .(\delta-1) =g . \gamma= \sum_{j=0}^{k_{n-1}} f_j . \gamma  \otimes  v_j.\gamma 
\end{equation}
which, by induction hypothesis belongs to $M^{(n-1)} \otimes V_{n-1}$.
Therefore, $f. \gamma$ belongs to $M^{(n)}.$

Since \eqref{inclu} holds for all $g \in M^{(n-1)}\otimes V_{n-1}$,  the $\G$-invariance of $M^{(n-1)}$ implies the $\G$-invariance of $M^{(n-1)} \otimes V_{n-1}.$ 

(ii)  We have $M^{(n-1)} \otimes V_{n-1}=\text{pr}^{-1}(\{0\}) \subset M^{(n)}.$ 
 
(iii) This is seen by unraveling the definition, which, as shown in (i) is well-founded. 
\end{proof}

A first result on the structure of the space of iterated invariants is provided by the following lemma. To state it we introduce some additional notation, for each $\G$-module $M$:
$$C^1(\G, M)=\{\alpha:\G \to M\}, \qquad C^1_c(\G, M)=\{ \alpha \in C^1(\G, M); \text{ $\alpha(\pi)=0$ for all parabolic $\pi \in \G$}\}$$
$$Z^1(\G, M)=\{\text{1-cocyles of $\G$ in $M$}\}, \qquad Z^1_c(\G, M)= Z^1(\G, M) \cap C^1_c(\G, M)$$
$$B^1(\G, M)=\{\text{1-coboundaries of $\G$ in $M$}\}, \qquad B^1_c(\G, M)= B^1(\G, M) \cap C^1_c(\G, M)$$
and
$$H^1(\G, M)=Z^1(\G, M)/B^1(\G, M), \qquad H^1_c(\G, M)=Z^1_c(\G, M)/B^1_c(\G, M).$$
With this notation we have:
\begin{lem}\label{n=2} Let $n \in \mathbb N$. There is a map $\psi$ such that the following sequence is exact:
$$0 \longrightarrow H^0(\Gamma, \bigotimes_{i=0}^{n-1}V_{i}) \overset{\iota}{\longrightarrow} M_c^{(n)} \overset{\psi}{\longrightarrow} M^{(n-1)} \otimes C^1_c(\G, V_{n-1}).$$
In particular for $n=2$ we have the exact sequence:
$$0 \longrightarrow H^0(\Gamma, V_0 \otimes V_{1})  \overset{\iota}{\longrightarrow} M_c^{(2)} \overset{\psi}{\longrightarrow}  
H^0(\Gamma, V_0) \otimes  Z^1_c(\G, V_1).$$
\end{lem}
\begin{proof}  Fix a basis $\{u_i\}$ of $M^{(n-1)}$.  Then, for every $f \in M^{(n)}_c$ and every $\gamma \in \G$, we have
$$f.(\gamma-1)=\sum \psi_i^f(\gamma) \otimes u_i.$$
for some $\psi_i^f(\g) \in V_{n-1}.$ By definition, each map $\g \to \psi_i^f(\gamma)$ gives an element of $C^1_c(\G, V_{n-1}).$ Therefore  the assignment $f \to \sum \psi_i^f \otimes u_i$ induces the map $\psi$ of the proposition. 

For the case $n=2,$ we note, with Prop. \ref{stru}(iii),  that $M^{(1)}=M^{(1)}_c=H^0(\G, V_0)$. Therefore, the $1$-cocycle condition satisfied by $\g \to f.(\g-1)$ is inherited by each $\psi_i^f \in C_c^1(\G, V_1)$.
\end{proof}
\begin{cor} \label{full} Let $\bar \psi$ be induced by $\psi$ and the natural projection $Z_c^1(\G, V_1) \to H^1_c(\G, V_1).$ Then we have the following exact sequence $$0 \longrightarrow H^0(\Gamma, V_0 \otimes V_{1})/(H^0(\Gamma, V_0)  \otimes H^0(\G, V_1))
 \overset{\bar \iota}{\longrightarrow} M_c^{(2)}/(M^{(1)} \otimes V_1^c) \overset{\bar \psi}{\longrightarrow}  
H^0(\Gamma, V_0) \otimes  H^1_c(\G, V_1) $$
where $\bar \iota$ is induced by $\iota$ and $V_1^c$ consists of $v \in V_1$ invariant under all parabolic $\pi \in \G.$
\end{cor}
\begin{proof} This is deduced directly from Lem. \ref{n=2}. It can also be deduced by the long exact sequence associated with 
$$0 \longrightarrow M^{(1)} \otimes V_{1}  \longrightarrow M^{(2)} \longrightarrow  M^{(2)}/(M^{(1)} \otimes V_{1})\longrightarrow 0.
$$
\end{proof}
\subsection{Extended higher order modular forms.}
For $k_0 \in \mathbb Z$ and positive even integers $k_1, k_2, \dots$, let $V=\mathfrak O=(\underset{2-k_i, 0}{|}, V_i)_{i \ge 0}$, where
$V_0=\mathcal O$ and $V_i=P_{k_i-2}[X_i]$ ($i \ge 1$) is the space of polynomials in $X_i$ of degree $\le k_i-2$.
We call the elements of $M_c^{(n)}(\mathfrak O)$ {\it extended modular forms of order $n.$}
With Prop. \ref{stru}(iii) we see that this is the space of $f(z; X_1, \dots X_{n-1}) \in \OO[X_1, \dots, X_{n-1}]$ such that
$$ f.(\g-1) \in  \begin{cases} \{0\} & (n =1) \\
M^{(n-1)}\otimes P_{k_{n-1}-2}[X_{n-1}] \, \, \,  & (n \ge 2) \end{cases}, \qquad \text{and, for all parabolic $\pi \in \G$, $f.\pi=f$}$$
where the action of $\G$ is induced by
\begin{equation}\label{.}(f.\g)(z;  X_1, \dots X_{n-1}):=f(\g z; \g X_1, \dots \g X_{n-1})j(\g, z)^{k_0-2}j(\g,X_1)^{k_1-2} \dots j(\g, z)^{k_{n-1}-2}
\end{equation}
In particular, 
\begin{equation}\label{class}
M_c^{(1)}(\mathfrak O)=M_{2-k_0}(\G)=\{\text{weight $2-k_0$ holomorphic modular forms for $\G$}\}.
\end{equation}
Let $V=\mathfrak O_c=(\underset{2-k_i, 0}{|}, V_i)_{i \ge 0}$, where
$V_0=\mathcal O_c$ and $V_i=P_{k_i-2}[X_i]$.
Then we obtain the space $M^{(n)}_c(\mathfrak O_c)$ {\it extended cusp forms of order $n.$}

{\bf Remark. }The adjective ``extended"  in the previous examples aims to distinguish them from the class of (standard) higher-order modular forms (see e.g. \cite{CDO}). We can retrieve the standard higher-order modular forms by setting $k_1=\dots=k_n=2.$ Then $.$ is simply $\underset{2-k_0, 0}{|}$ for all $n$ and the space $M^{(n)}_c(\mathfrak O)$ consists of all $f \in \mathcal O$ such that, for all $\g \in \G_0(N)$ and for all parabolic $\pi \in \G_0(N)$,
$$f \underset{2-k_0, 0}{|}(\g-1) \in M^{(n-1)}_c(\mathfrak O) \qquad \text{and $f\underset{2-k_0, 0}{|} \pi  =f.$}$$
The standard higher-order modular forms become trivial in $\G_0(1)$ because, as shown in \cite{CDO}, they are parametrised by weight $2$ cusp forms which are trivial in SL$_2(\mathbb Z).$
(Also note that, in contrast to general iterated invariants, $M^{(n-1)}_c(\mathfrak O)=M^{(n-1)}(\mathfrak O)$ because of the identity $(\g-1)(\pi -1)=(\g \pi \g^{-1}-1)\g-(\pi -1).$)

\subsubsection{Iterated Eichler integrals}
Important examples and, indeed, some of the prototypes, of the standard higher order forms mentioned in the closing Remark of the last section are the antiderivatives of weight $2$ cusp forms and their higher iterated analogues:
$$\int_{i \infty}^z f(w)dw, \, \, \int_{i \infty}^z f(w_1)\int_{i \infty}^{w_1}g(w_2) \dots dw_2 dw_1 \qquad \text{for weight $2$ cusp forms $f, g\dots $}$$

We will show that, more generally, Eichler integrals and their iterated counterparts belong to $M_c^{(n)}(\mathfrak O)$. This is summarised in the following lemma:
\begin{lem}\label{Eichl} Let $k_0=2$ and $k_1, \dots k_{n-1} \in 2 \mathbb N.$ Suppose that, for $i=1, \dots, n-1,$ $f_i$ is a weight $k_i$ cusp form for \emph{SL}$_2(\mathbb Z)$. Let $F_n
\in \mathcal O [X_1, \dots, X_{n-1}]$ 
be defined by $F_1=1$ and, for $n \ge2,$
$$F_n(w; X_1, \dots X_{n-1}):=\int_{i \infty}^w f_1(w_1)(w_1-X_1)^{k_1-2} \int_{i \infty}^{w_1} f_2(w_2)(w_2-X_2)^{k_2-2}  \dots 
dw_{n-1} \dots dw_1$$
Then $F_n \in M^{(n)}_c(\mathfrak O)$. 
\end{lem}
\begin{proof} We first show the assertion for $n=2$.  

The action $.$ is given explicitly by \eqref{.}.  We then see that $F_2.\g$ is
$$\int^{\g w}_{i \infty} f_1(w_1)(w_1- \g X_1)^{k_1-2} j(\g, X_1)^{k_1-2} dw_1=\int_{\g^{-1} i \infty}^{w} f_1(w_1)(w_1-X_1)^{k_1-2} dw_1$$
where the last integral is obtained by a change of variables and \eqref{gzgX}. Therefore, with \eqref{class}
$$F_2.(\g-1)= \int_{\g^{-1} i \infty}^{i \infty} f_1(w_1)(w_1-X_1)^{k_1-2} dw_1 \in \mathbb C \otimes P_{k_1-2}[X_1] \subset M^{(1)}(\mathfrak O) \otimes P_{k_1-2}[X_1].$$
The same identity shows that $F_2. T=F_2$ and, hence, $F_2 \in M^{(2)}_c(\mathfrak O).$

The proof for general $n$ is an application of the shuffle product formula for iterated integrals, but we give a direct proof by induction. As before, by the definition of the action $.$ on $\mathcal O [X_1, \dots, X_{n-1}]$, the changes of variables $w_i \to \g w_i$ and \ref{gzgX} we deduce that, for $n>2$,  $F_n .\g$ equals
\begin{multline}
\int_{\g^{-1} i \infty}^{w} f_1(w_1)(w_1-X_1)^{k_1-2} \left ( \int_{\g^{-1} i \infty}^{w_1} f_2(w_2)(w_2-X_2)^{k_2-2}  \dots dw_2 \right ) dw_1=\\
\left ( \int_{ i \infty}^{w}+\int_{\g^{-1} i \infty}^{i \infty} \right )  f_1(w_1)(w_1-X_1)^{k_1-2} \left ( \left ( \int_{ i \infty}^{w_1}+\int_{\g^{-1} i \infty}^{i \infty} \right )   f_2(w_2)(w_2-X_2)^{k_2-2}  \dots dw_2 \right ) dw_1
\end{multline}
where the sum of integral signs indicates that they are both applied to the integrand following them. Therefore $F_n.(\g-1)$ is a sum of iterated integrals such that each iterated integral includes at least one constituent integral with limits $i \infty$ and $\g^{-1} i \infty$, e.g.
$$\int_{ i \infty}^{w} f_1(w_1)(w_1-X_1)^{k_1-2}  \int_{\g^{-1} i \infty}^{i \infty} f_2(w_2)(w_2-X_2)^{k_2-2} 
\int_{ i \infty}^{w_2} f_3(w_3)(w_3-X_3)^{k_3-2} \dots dw_3 dw_2 dw_1.$$
This, on the one hand, implies that $F_n.(T-1)=0$ and, on the other, that  $F_n.(\g-1)$ is a sum of products of the form
$F_{i-1} \cdot P_{i-1}(X_{i-1}, \dots X_{n-1})$ ($i=2, \dots n$), where the polynomials $P_{i-1}$ are of degree $\le k_j-2$ in $X_j$ ($j= i-1, \dots, n-1$) and independent of $w$ . By induction, each product 
$F_{i-1} \cdot P_{i-1}(X_{i-1}, \dots X_{n-1})$ belongs to 
$$M^{(i-1)}(\mathfrak O) \otimes P_{k_i-1}[X_i] \otimes \dots P_{k_{n-1}-2}[X_{n-1}] \subset
M^{(n-1)} (\mathfrak O) \otimes  P_{k_{n-1}-2}[X_{n-1}].$$ 
This completes the proof of the statement.
\end{proof} 
Upon tensoring with a space of cusp forms, we obtain:
\begin{cor} Let $k_0, \dots ,k_{n-1} \in 2 \mathbb N.$ Suppose that, for $i=1, \dots, n-1,$ $f_i$ is a weight $k_i$ cusp form for SL$_2(\mathbb Z)$ and that $f_0$ is a cusp form of weight $2-k_0$. Let $F_n 
\in \mathcal O [X_1, \dots, X_{n-1}]$ 
be defined by
$$f_0(w)\int_{i \infty}^w f_1(w_1)(w_1-X_1)^{k_1-2} \int_{i \infty}^{w_1} f_2(w_2)(w_2-X_2)^{k_2-2}  \dots 
dw_{n-1} \dots dw_1$$
Then $F_n \in M^{(n)}_c(\mathfrak O_c)$. 
\end{cor}

\subsection{The space of real-analytic iterated integrals}
\label{HOFS1}
We now associate iterated invariants to real analytic modular forms. This will lead to a way to obtain invariant objects from iterated integrals associated with modular forms. 

Let $V=(V_i, \rho_i)_{i \ge 0}$ where $V_0=\mathcal R$ (resp. $V_0=\mathcal R_c$ ) with $\G=$SL$_2(\mathbb Z)$ acting though $\underset{r, s}{|}$, and $V_i=P_{k_i-2}[X_i]$ $(i \ge 1)$ with the usual action $\underset{2-k_i, 0}{|}$ on polynomials. Then we have that, for $i>0$,
\begin{align}
&V_i^c=H^0(\G, V_i)=H^0(\G, P_{k_i-2})=0 \qquad \text{if $k_i>2$ (by translation invariance), and} 
\label{coho:1} 
\\
&H^1_c(\G, V_i) \cong S_{k_i} \oplus \bar S_{k_i} \quad \text{(by Eichler-Shimura combined with Lem. 1 of VI \S 5 of \cite{La})} 
\label{coho:2}
\end{align}
Notice that although $H^0(\G, \mathcal R)$ is very similar to $\mathcal M_{r, s}$ they are not the same, because the former includes functions that do not have a Fourier expansion of the form \eqref{FE}.

We also consider the one-dimensional subspace $\tilde{\mathcal M}_{r, s}$ of $\mathcal M_{r, s} \subset H^0(\G, \mathcal R)$ generated by $\mathcal E_{r, s}$. 
With the map $\bar \psi$ defined in Cor. \ref{full}, we set 
$$\tilde M_c^{(2)}(\mathcal R):=
\bar \psi^{-1}( \tilde{\mathcal  M}_{r, s} \otimes H^1_c(\Gamma, V_1)) $$ 
Explicit examples of elements of $\tilde M_c^{(2)}(\mathcal R)$ are certain real-analytic analogues of iterated Eichler integrals in the case of $\mathcal E_{r, s}.$ We will make this more specific in the case $n=2$ with the following corollary of Lemma \ref{Eichl}:
\begin{cor}\label{EichlErs} Suppose that $k_1 \in 2 \mathbb N$ and that $f_1$ is a weight $k_1$ cusp form for SL$_2(\mathbb Z)$. Let $F_2 
\in \tilde{\mathcal M}_{r, s} [X_1]$ be defined by
$$F_2(w, X_1):=\mathcal E_{r, s}(w)\int_{i \infty}^w f_1(w_1)(w_1-X_1)^{k_1-2} dw_1.$$
Then $F_2 \in \tilde M_c^{(2)}(\mathfrak R).$
\end{cor}
\begin{proof} The corollary follows immediately from the identity
$F_2.(\g-1)= \mathcal E_{r, s} r_f(\gamma; X_1)$ and the definition of $\bar \psi$.
\end{proof}
Since the functions of this corollary have been the prototypes for the elements of $\tilde M_c^{(2)}(\mathfrak R)$, we refer to it as the space of \emph{real-analytic iterated integrals}.

Now, with the definition of $\tilde M_c^{(2)}(\mathcal R)$ and \eqref{coho:1}, \eqref{coho:2},
Cor. \ref{full} becomes
$$0 \longrightarrow H^0(\Gamma, \mathcal R \otimes P_{k_1-2}[X_1])
 \overset{\bar \iota}{\longrightarrow} \tilde M_c^{(2)}(\mathcal R)
\overset{\bar \psi}{\longrightarrow}  
\tilde{\mathcal M}_{r, s} \otimes  (S_{k_1} \oplus \bar S_{k_1}) $$
We will show that this can be completed to a right exact sequence. 
\begin{thm}\label{1stclassif} Suppose that $r+s>k_1$. The sequence of maps 
$$0 \longrightarrow H^0(\Gamma, \mathcal R \otimes P_{k_1-2}[X_1])
 \overset{\bar \iota}{\longrightarrow} \tilde M_c^{(2)}(\mathcal R)
\overset{\bar \psi}{\longrightarrow}  
\tilde{\mathcal M}_{r, s} \otimes  (S_{k_1} \oplus \bar S_{k_1}) \longrightarrow 0$$
is exact.
\end{thm}
\begin{proof} The only part remaining to be proved is the surjectivity of $\bar \psi$. Let $\mathcal E_{r, s} \otimes (f, \bar g)$ be an arbitrary basis element of $\tilde{\mathcal M}_{r, s} \otimes  (S_{k_1} \oplus \bar S_{k_1})$. With the notation of sub-section \ref{explsubc}, assign to each $h \in S_{k_1}$ a function $\psi^{\pm}_{h; r, s}$ given by 
$$\psi_{h; r, s}^{\pm}(z, X):=\phi^{\pm}_{r, s}(h; z, X)- F^{\pm}_h(z, X) \mathcal E_{r, s}=\sum_{\g \in B \backslash \G} \frac{
r^{\pm}_h(\g; X)}{j(\gamma, z)^{r} j(\gamma, \bar z)^{s}}.$$
By Prop. \ref{phi}, this is absolutely convergent and its coefficients are of polynomial growth at infinity, and thus, they belong to $\mathcal R.$
 
The image of $\psi^+_{f; r, s}+\psi^-_{g; r, s}$ under $\psi$ is induced by the mapping 
\begin{align}\label{psiDef}
\g \to (\psi^+_{f; r, s}+\psi^-_{g; r, s}) \underset{r, s, 2-k}{|}(\g-1)&=-(F^+_f+F^-_g)|_{0, 0; 2-k} \mathcal E_{r, s}(z)
\nonumber
\\
&=-(r_f(\g; X)+\overline{r_g(\g; X)} )\mathcal E_{r, s}(z).
\end{align}
For the two equalities we have used Prop. \ref{phi} and \eqref{decomp:1}.
By the explicit formula for the Eichler-Shimura map we deduce that $\bar \psi(-\psi^+_{r, s}-\psi^-_{r, s})=\mathcal E_{r, s} \otimes (f, \bar g).$ This shows that $-\psi^+_{r, s}-\psi^-_{r, s} \in \tilde{ M}^{(2)}_c(\mathcal R)$ and that its image is the element  $\mathcal E_{r, s} \otimes (f, \bar g)$.
\end{proof}

{\bf Remark.} The theorem could be stated in more general form so that the real-analytic analogues of both Eisenstein and Poincare series are captured. That would have the advantage of accounting for the full space $M_c^{(2)}$ instead of $\tilde M_c^{(2)}$, but we would need to enlarge our investigations to objects that do not satisfy \eqref{FE}. This is because the ``real-analytic Poincare series" do not satisfy \eqref{FE}. However, they are are clearly interesting objects, worthwhile studying, which are the subject of work in progress with F. Str\"omberg.

The family $\{\psi_{h; r, s}^{\pm}\}$ constructed in Th. \ref{1stclassif} allows us to describe our approach to Question 2 of the Introduction. Specifically, for $h \in S_{k_1}$, set
$$\psi_{h; r, s}^{\pm}(z, X):=\sum_{\g \in B \backslash \G} \frac{
r^{\pm}_h(\g; X)}{j(\gamma, z)^{r} j(\gamma, \bar z)^{s}}.$$
The family addresses Question 2, inasmuch as it satisfies the following three properties:

Firstly, by Th. \ref{1stclassif}, $\psi_{h; r, s}^{\pm}$ belong to the space $\tilde{ M}^{(2)}_c(\mathcal R)$ of real-analytic iterated integrals.

Secondly, this family is ``canonical" in the sense that it induces a generating set for $\tilde{\mathcal M}_{r, s} \otimes  (S_{k_1} \oplus \bar S_{k_1}).$

Thirdly, it is possible to obtain, by a simple process, explicit $\G$-equivariant versions of the real-analytic iterated integrals $\psi_{h; r, s}^{\pm}$.  
 This process is given in the following proposition which also formalises a link between the two main themes of this note, namely second-order modular forms and iterated integrals:
\begin{prop} Let $r+s>k_1.$ There is a well-defined linear map from the subspace of $\mathcal M_c^{(2)}(\mathcal R)$ generated by the family $\{\psi_{h; r, s}^{\pm}\}$ to $\oplus_{i=0}^{k_1-2} \MI_2'$.
\end{prop}
\begin{proof} For each $\psi_{h; r, s}^{\pm}(z, X)$ consider 
$$\phi_{r, s}^{\pm}(h; z, X)=\psi_{h; r, s}^{\pm}(z, X)+ F^{\pm}_h(z, X) \mathcal E_{r, s}.
$$ By Th. \ref{mathcalA}, the coefficients $\phi^{\pm}_{r, s}(h; i, z, X)$ of $(X-z)^i(X-\bar z)^{k_1-2-i}$ in  $\phi_{r, s}^{\pm}(h; z, X)$ belongs to $\MI_2$. Therefore, the assignment 
$$\psi_{h; r, s}^{\pm}(z, X) \to (\phi^{\pm}_{r, s}(h; 0, z, X), \dots, \phi^{\pm}_{r, s}(h; k-2, z, X))$$
defines the sought map.  
\end{proof}

\subsubsection{A classification of holomorphic second-order forms}
We close with another implication of Cor. \ref{full} to \emph{holomorphic} extended second-order modular forms. The resulting theorem is the analogue of Th. 2.3 of \cite{CDO} to the class of extended second-order modular forms.

Let $\rho$ be a complex representation of $\G$ of dimension $k_1-1$ induced by 
$$\rho(-I_2)=I_{k_1-1}, \, \, \rho(S)=\left ( (-1)^i \delta_{j, k_1-2-i}\right )_{i, j=0}^{k_1-2} \qquad \text{and} \, \, 
\rho(T)=\left ( (-1)^{i+j} \binom{j}{i} \right )_{i, j=0}^{k_1-2}.
$$
For an even $k>0$, we denote by $M_k(\rho)$ the space of vector-valued modular forms for the representation $\rho.$
\begin{prop} Suppose that $k>k_1>2.$The sequence of maps
$$0 \longrightarrow M_k(\rho)
 \longrightarrow M_c^{(2)}(\mathfrak O) \overset{\bar \psi}{\longrightarrow}  
M_k \otimes  (S_{k_1} \oplus \bar S_{k_1}) \longrightarrow 0$$
is exact. In particular,
$$\emph{dim} M_c^{(2)}(\mathfrak O)=2\emph{dim}(M_k)\emph{dim}(S_{k_1})+ \frac{5+k}{12}(k_1-1)+\frac{i^{k+k_1-2}}{4}-\frac13 \legendre{k_1-1}{3} \legendre{k-1}{3}$$
\end{prop}
\begin{proof}
We apply Cor. \ref{full} to $V=\mathfrak O$, with $k_0:=2-k$ to deduce, in the first instance,
$$0 \longrightarrow H^0(\Gamma, \mathcal O \otimes P_{k_1-2})
 \overset{\bar \iota}{\longrightarrow} M_c^{(2)}(\mathfrak O) \overset{\bar \psi}{\longrightarrow}  
M_k \otimes  (S_{k_1} \oplus S_{k_1}) .$$
Here we used \eqref{coho:1} and \eqref{coho:2}.

The isomorphism $H^0(\Gamma, \mathcal O \otimes P_{k_1-2}) \cong M_k(\rho)$ is induced by the mapping
$$g=\sum_{j=0}^{k_1-2} f_j(z) X^j \longrightarrow (f_0, f_1, \dots, f_{k_1-2})^T.$$
 By a direct computation, we see that the invariance of $g$ under $S$ and $T$ is equivalent to the invariance of
the associated vector under $S$ and $T$ in terms of the representation $\rho.$

To prove the surjectivity, let $f \otimes (g, \bar h) \in M_k \otimes (S_{k_1} \oplus \bar S_{k_1})$. Assume that $f=\sum_{n \ge 0} \lambda_n P_n,$ where $P_n$ is the classical Poincar\'e series of weight $k$ given by
$$P_n(z)=\sum_{B \backslash \G} \frac{e^{2 \pi i n \g z}}{j(\g, z)^k}.$$
(Here $\lambda_n=0$, for all but finitely many integer $n \ge 0.$)
For $n \ge 0$ and $h \in S_{k_1}$, set
$$G_{n, h}^{\pm}(z, X):=\sum_{\g \in B \backslash \G} (F_f^{\pm} \cdot e^{2 \pi i n -})  
\underset{k,0, 2-k_1}{\big |} \g
- F^{\pm}_h(z, X) P_n=\sum_{\g \in B \backslash \G} \frac{
r^{\pm}_h(\g; X)e^{2 \pi i n \g z}}{j(\gamma, z)^{k}}.$$
Since $|e^{2 \pi i n \g z}| \le 1$, the absolute convergence of this series, for $k>k_1$ follows from the absolute convergence of $\phi^{\pm}_{k, 0}$ proved in Prop. \ref{phi}. The same proposition implies the polynomial growth of the polynomial coefficients of $\phi^{\pm}_{k, 0}$.
 
We then have, for all $\g \in \G,$
$$\sum_{n \ge 0} \lambda_n(G_{n, g}^{+}+G_{n, h}^-)\underset{k,0, 2-k_1}{\big |} (\g-1)=\sum_{n \ge 0} \lambda_n(r^+_g(\g; X)+r^-_h(\g; X)) P_n=(r^+_g(\g; X)+r^-_h(\g; X))f.
$$
By the definition of the map $\bar \psi$, we deduce that the image of $\sum \lambda_n (G_{n, g}^++G_{n, h}^-)$ is $f \otimes (g, \bar h)$.

 To deduce the dimension formula, we use Theorem 6.3 and Remark 6.4 of \cite{CF} to compute the dimension of $M_k(\rho)$. Indeed, $\rho$ is an even representation, dim$\rho=k_1-1$ and a direct computation gives
$$\text{Tr}(\rho(ST)^2)=\text{Tr}(\rho(ST))=\text{Tr}(\rho(S)\rho(T))=\sum_{i=0}^{k_1-2}(-1)^i \binom{i}{k_1-2-i}
\legendre{k_1-1}{3}$$
where $\legendre{\cdot}{\cdot}$ stands for the Legendre symbol. This is seen by noting that the sequence given by
$$a_n=\sum_{i, j \in \mathbb Z; i+j=n} (-1)^i\binom{i}{j}; \qquad \text{where $\binom{i}{j}=0$, unless $i \ge j \ge 0$}$$
satisfies the recurrence relation $a_n+a_{n-1}+a_{n-2}=0$ and, therefore, by induction, $a_n=\legendre{n+1}{3}$

We can also see, by induction, that, for $\xi=e^{\pi i /3}$ 
$$\frac{\xi^k}{1-\xi^2}+\frac{x^{2k}}{1-\xi^{-2}}=-\legendre{k-1}{3}.$$ Since the only eigenvalue of $\rho(T)$ is $1$ and the corresponding eigenspace has dimension $1$. Therefore there is one Jordan block with $1$ in the diagonal and, by  Th. 3.4. of \cite{CF}, we deduce that the trace for a standard choice of exponents for $\rho (T )$ is $0$. This, with Remark 6.3 of \cite{CF}, implies that, since $k>k_1$, the dimension of $M_k(\rho)$ is given by the formula of Th. 6.3 of \cite{CF}, which, by the preceeding remarks is
$$\frac{5+k}{12}(k_1-1)+\frac{i^{k+k_1-2}}{4}-\frac13 \legendre{k_1-1}{3} \legendre{k-1}{3}.$$
\end{proof}

\end{document}